\documentclass{problemset}
\studname{Rudra Kamat}
\studmail{rrk48@cornell.edu}
\coursename{\textbf{PHYS 6554:}}
\duedate{\today}
\hwNo{0}
\newtheorem{rem}{Remark}
\begin{document}
\maketitle

\begin{abstract}
Williamson's theorem is well known for symmetric matrices. In this paper, we state and re-derive some of the cases of Williamson's theorem for symmetric positive-semi definite matrices and symmetric matrices having negative index 1, due to H\"{o}rmander. We prove theorems that guarantee conditions under which two symmetric positive-definite matrices can be simultaneously diagonalized in the sense of Williamson's theorem and their corollaries. Finally, we provide an application of this result to physical systems and another connecting the decompositions for the degenerate and non-degenerate cases, involving phase space constraints that we later apply to phase space cylinders and ellipsoids via symplectic capacities.
\end{abstract}

\section{Introduction}
Let $V$ be a $2n$-dimensional real vector space equipped with a symplectic form $\omega$. Such a $(V,\omega)$ is called a \textit{symplectic vector space}.
Suppose that this symplectic vector space is also equipped with an inner product $\langle\cdot,\cdot\rangle$.
Let $\mathbb{M}(V)$ denote the set of real linear automorphisms on $V$. An element $\op{M}\in\mathbb{M}(V)$ is called symmetric if it is symmetric with respect to the inner product i.e. $\langle v,\op{M}w\rangle=\langle \op{M}v,w\rangle$. \\
Given a symmetric operator $\op{M}\in\mathbb{M}(V)$, by $\op{M}>0$ and $\op{M}\geq 0$ we imply that $\op{M}$ is positive-definite and positive-semi definite respectively.
Let $\op{I}$ denote the identity linear map on $V$.
\begin{rem}\label{re1}
Define $\op{J_{2n}}=\begin{bmatrix}\op{0_n} & \op{I_n}\\ -\op{I_n} & \op{0_n} \end{bmatrix}\in\mathbb{M}(\R^{2n})$, \\where $\op{I_n}$ and $\op{0_n}$ are the $n\times n$ identity and zero matrix respectively.

The anti-symmetric bilinear form $\omega(\cdot,\cdot)$ on $\R^{2n}$ given by
\begin{equation}
    \omega(x,y)\coloneqq\sum_{i=1}^{n}(x_iy_{n+i}-x_{n+i}y_i)
\end{equation}
is called the \textit{symplectic inner product} on $\R^{2n}$. This can also be expressed in terms of the Euclidean inner product as $\omega(x,y)=\langle x,\op{J_{2n}} y\rangle=x^T\op{J_{2n}}y$.
\end{rem} 

Let us define the \textit{orthogonal group} on $V$ as follows:
$$\op{O}(V)\coloneqq\{\op{U}\in\mathbb{M}(V)\;|\;\langle\op{U}x,\op{U}y\rangle=\langle x,y\rangle\}$$ 
An element of $\op{U}\in \op{O}(V)$ is called a \textit{orthogonal} linear automorphism on $V$. We can see from the above definition that a orthogonal linear automorphism on $V$ is such that it preserves the inner product on $V$.\\
Similarly, we can define the \textit{symplectic group} on $V$ as follows:
$$\op{Sp}(V)\coloneqq\{\op{S}\in\mathbb{M}(V)\;|\;\omega(\op{S}x,\op{S}y)=\omega(x,y)\}$$ 
An element of $\op{S}\in \op{Sp}(V)$ is called a \textit{symplectic} linear automorphism on $V$. We can see from the above definition that a symplectic linear automorphism on $V$ is such that it preserves the symplectic product on $V$. Elements of $\mathbb{M}(V)$ that are both symplectic as well orthogonal automorphisms are called \textit{orthosymplectic} and they form the \textit{orthosymplectic group} on $V$, $\op{OrSp}(V)$.

The term \textit{symplectic}, introduced by H. Weyl \cite{wey}, is a calque of \textit{complex}; previously, the symplectic group had been called the \textit{line complex group}. \textit{Complex} comes from the Latin \textit{com-plexus}, meaning \textit{braided together} (co- + plexus), while symplectic comes from the corresponding Greek \textit{sym-plektikos} $(\sigma\upsilon\mu\pi\lambda\epsilon\kappa\tau\iota\kappa$\'{o}$\zeta)$; in both cases the stem comes from the Indo-European root $^*$\textit{ple}$\invbreve{\textit{k}}$ -.\\The name reflects the deep connection between complex and symplectic structures.

The reason for this is because we can define a linear map $\op{J}$, called a \textit{complex structure} from $V$ to itself, satisfying $\op{J}^2=-\op{I}$. This complex structure serves the same role on a real vector space that the imaginary unit $i$ does on the real numbers. For example if $z=a+ib$ is a complex number and $v$ a vector in $V$, then we can define the complex scalar multiple $zv=av+b\op{J}v$. This makes the real vector space $V$ into a complex vector space.\\
In Remark \ref{re1}, $\op{J}_{2n}$ is called the standard complex structure on $\R^{2n}$. 

J. Williamson first discovered, in 1936, a result involving the decomposition of a symmetric matrix by a symplectic matrix in \cite{Will}\\
However, only the special case of the result, when the symmetric matrix is positive-definite, is widely known in literature as \textit{Williamson's theorem}.\\
D.M. Galin \cite{gal}, in 1982, and then A.D. Bryuno \cite{bry}, in 1988, only restated but did not re-derive Williamson's general result in their individual works.\\
Only L. Hörmander \cite{Horm} stated and proved Williamson's theorem for the cases when the symmetric matrix is positive semi-definite and when it is hyperbolic (with one negative index).\\
One possible explanation as to why only the positive-definite case of Williamson's theorem is given in the final remarks section of F. Nicacio's work \cite{nic} where it is stated that the reason the positive-definite case is so practical and celebrated can be attributed to the fact that the normal form attained after decomposition is diagonal for this case and therefore easier to work with. Also, from an application standpoint, not many of the other cases form thermodynamically stable systems unlike the positive-definite case.

In the last decade, although there were some works on the topic before 2014, it was all scattered, until T. Jain and R. Bhatia released a paper in 2015 \cite{bhatia2015} giving several interesting connections between symplectic eigenvalues and eigenvalues, and symplectic matrices and orthogonal matrices. Following that work, T. Jain and H.K. Mishra \cite{jm} provided a comprehensive study on qualitative and quantitative properties of symplectic eigenvalues. These may be regarded as one of the first papers, in recent times, to treat symplectic eigenvalues and symplectic matrices in detail. Following these works, there have been numerous studies on the topic (see \cite{mishra2020first}, \cite{jain2021sums}, \cite{bhatia2021variational} and the references therein).\\
The most recent works in symplectic eigenvalues are by H.K. Mishra \cite{mishra2024equality}, G. Babu and H.K. Mishra \cite{mishra2023}. The most recent works in symplectic decomposition are by H.K. Mishra \cite{mishra2024orthosymplectic}, Son et al. \cite{willdeg1} and N.T. Son and T. Stykel \cite{son}.


\section{Background and notation}\label{back}
Let us now build some notation, terminology and results that we will need going forward.

A pair of vectors $\{u,v\}$ is called \textit{symplectically normalized} if $\omega(u,v)=1$. Two pairs of vectors $\{u_1,v_1\}$ and $\{u_2,v_2\}$ are called \textit{symplectically orthogonal} if 
$$\omega(u_i,v_j)=\omega(u_i,u_j)=\omega(v_i,v_j)=0$$
for all $i\neq j$; $i,j=1,2$. A subset $\{u_1,\dots,u_m,v_1,\dots,v_m\}$ of $V$ is called a \textit{symplectically orthogonal (orthonormal)} set if the pairs of vectors $\{u_j,v_j\}$ are mutually symplectically orthogonal (and normalized). If $m=n$, then the symplectically orthonormal set is called a \textit{symplectic basis} of $V$.

\vspace{0.4cm}
Let $W$ be a vector subspace of the symplectic vector space $(V,\omega)$ then the \textit{symplectic complement} of $W$ is defined as:
$$W^\omega\coloneqq\{x\in V\;|\;\omega(x,y)=0;\;\forall y\in W\}$$
It is easy to see that $W^\omega$ is also a subspace of $V$. Depending on the relation between $W$ and $W^\omega$, we have the following classification:
\begin{enumerate}
    \item If $W\subset W^\omega$, then $W$ is an \textit{isotropic subspace} of $V$
    \item If $W^\omega\subset W$, then $W$ is a \textit{coisotropic subspace} of $V$
    \item If $W\cap W^\omega=\{0\}$, then $W$ is a \textit{symplectic subspace} of $V$
    \item If $W=W^\omega$, then $W$ is a \textit{Lagrangian subspace} of $V$
\end{enumerate}
Analogous to Euclidean vector spaces, we have a the following dimension theorem for our symplectic vector space $(V,\omega)$:
$$dim(V)=2n=dim(W)+dim(W^\omega)$$
We can also define a symplectic subspace $W$ of $V$ by requiring that for every $x\in W\setminus\{0\}$, there exists a $y\in W$ such that $\omega(x,y)\neq 0$. \\
This helps us show that we can always construct a symplectic basis for a symplectic subspace analogous to the \textit{Gram-Schmidt process} for Euclidean vector spaces  and conversely, a given subspace is symplectic if it owns a symplectic basis (see \cite{Birk}, Theorem 1.15).

\vspace{0.4cm}
Let $(V,\omega)$ be a $2n$-dimensional symplectic vector space equipped with a positive-definite quadratic form $f$.
To the quadratic form $f$ corresponds a unique symmetric bilinear form $A$ on $V$ such that $f(v)=A(v,v)$ for all $v$; it is given by the usual formula $A(v,w)=\frac{1}{2}(f(v+w)-f(v)-f(w))$. Since $\omega$ is non-degenerate, there exists a unique linear map $\op{F}:V\to V$ such that 
$$A(v,w)=\omega(v,\op{F}w)$$
for all $v,w\in V$. It follows from the symmetry of $A$ and the anti-symmetry that $\omega$ that $\omega(v,\op{F}w)=-\omega(\op{F}v,w)$ for all $v,w\in V$. That means, $\op{F}$ is anti-symmetric with respect to $\omega$ i.e. $\op{F}$ is an element of the Lie algebra of the group of symplectic linear transformations of $V$.
\begin{rem}\label{re1} 
Let $V=\R^{2n}$ with its standard symplectic form, which is given by $\omega(v,w)=v^T\op{J_{2n}}w$. Any positive-definite bilinear form $A$ can be written as $A(v,w)=v^T\op{M_A} w$, where $\op{M_A}$ is a positive-definite symmetric $2n\times 2n$ matrix. The matrix of the linear map $\op{F}$ is given by $\op{J_{2n}M_A}$.
\end{rem}
Let $V^{\C}$ be the complexification of $V$ and for $\lambda\in\C$, let $V_{\lambda}$ be the subspace of $V^{\C}$ consisting of all $v\in V^{\C}$ with $\op{F}v=\lambda v$. The symplectic form $\omega$ extends to a complex symplectic form on $V^{\C}$, which we shall denote by $\omega$. The bilinear form $A$ extends to a complex bilinear form on $V^{\C}$, which we denote also by $A$. The real linear map $\op{F}:V\to V$ extends to a complex linear map $V^{\C}\to V^{\C}$ which we also denote by $\op{F}$.\\
Note that if $\{a_1,\dots,a_n,b_1,\dots,b_n\}$ forms a symplectic basis of $V$ then, with the appropriate normalization, 
$\{(a_1+ib_1),\dots,(a_n+ib_n),(a_1-ib_1),\dots,(a_n-ib_n)\}$ also forms a symplectic basis of the complexification $V^{\C}$.

\begin{lemma}\label{lem1}
Let $v\in V_\lambda$ and $w\in V_\mu$. If $\lambda+\mu\neq 0$, then $\omega(v,w)=0$.
\end{lemma}
\begin{proof}
Let $\op{I}$ be the identity map of $V$. The linear map $\op{F}+\mu\op{I}$ leaves the subspace $V_\lambda$ invariant. Since $-\mu\neq\lambda$, then $\op{F}+\mu\op{I}$, considered as a map from $V_\lambda$ to $V_\lambda$, is bijective. Therefore, there exists a unique $u\in V_\lambda$ such that $(\op{F}+\mu\op{I})u=v$. Using anti-symmetry of $\op{F}$ and $(\op{F}-\mu\op{I})w=0$ gives \\
$\omega(v,w)=\omega((\op{F}+\mu\op{I})u,w)=\omega=\omega(u,(-\op{F}+\mu\op{I})w)=\omega(u,0)=0$
\end{proof}

\begin{lemma}\label{lem2}
All eigenvalues of $\op{F}$ are non-zero and purely imaginary.
\end{lemma}
\begin{proof}
Since $f$ is positive-definite, $\op{F}$ is invertible and therefore has no zero eigenvalues. Suppose there existed $\lambda\in\C$ with $\Re \lambda\neq 0$ and $V_\lambda\neq \{0\}$. For $v\in V^{\C}$, set $\Re v=\frac{1}{2}(v+\overline{v})\in V$. Then 
$$W_\lambda=\Re(V_\lambda)=\{\Re v\;|\;v\in V_\lambda\}$$
is a linear subspace of $V$. We claim that \begin{enumerate}\item\label{i1}$W_\lambda\neq \{0\};$\item\label{i2}$A(x,y)=0$ for all $x,y\in W_\lambda$ \end{enumerate}
Assume that $\Re v=0$ for all $v\in V_\lambda$; then $v+\overline{v}=0$ for all $v\in V_\lambda$. But if $v\in V_\lambda$, then $iv\in V_\lambda$, so $iv+\overline{iv}=v-\overline{v}=0$ and hence $v=0$. This contradicts the assumption $V_\lambda\neq \{0\}$, thus proving \ref{i1}. Now let $x, \op{F}y\in W_\lambda$. Writing $x=\Re u$ and $\op{F}y=\Re v$ with $u, v\in V_\lambda$, we find
\begin{align*}
A(x,y)=\omega(x,\op{F}y)&=\frac{1}{4}\omega(u+\overline{u},v+\overline{v})\\
&=\frac{1}{4}\omega(u, v)+\frac{1}{4}\omega(u,\overline{v})+\frac{1}{4}\omega(\overline{u}, v)+\frac{1}{4}\omega(\overline{u},\overline{v})
\end{align*}
We have $\overline{u}, \overline{v}\in V_{\bar{\lambda}}$ because $\op{F}$ is a real linear map. Since $\Re \lambda\neq 0$, we have $\lambda+\bar{\lambda}\neq 0$, so 
$\omega (u, \overline{v})=\omega(\overline{u}, v)=0$ by Lemma \ref{lem1}. Similarly, $\omega(u, v)=\omega(\overline{u}, \overline{v})=0$ by Lemma \ref{lem1}, because $\lambda \neq 0$. Thus, $A(x, y)=0$, proving \ref{i2}. The assertions $\ref{i1}-\ref{i2}$ contradict the positive definiteness of $f$, thus proving that $\Re \lambda = 0$.
\end{proof}
\begin{lemma}\label{lem3}
Let $v\in V_\lambda$ be an eigenvector of $F$ for eigenvalue $\lambda=i\mu$ with $\mu>0$. Write $v=x+iy$ with $x, y\in V$. Then $\op{F}x=-\mu y$, $\op{F}y=\mu x$ and $\omega(x,y)<0$.
\end{lemma}
\begin{proof}
Substituting $v=x+iy$ into $\op{F}v=i\mu v$ and equating real and imaginary parts gives $\op{F}x=-\mu y$ and $\op{F}y=\mu x$. On one hand,
\begin{align*}
A(\overline{v}, v) = A(x-iy, x+iy) = A(x, x) - A(iy, iy) = A(x, x) + A(y, y)
\end{align*}
and on the other hand
\begin{align*}
A(\overline{v}, v) &=\omega(\overline{v}, \op{F}v) =\omega(\overline{v}, i\mu v) =i\mu \omega(\overline{v}, v)\\
&=i\mu\omega(x-iy, x+iy)=-2\mu \omega(x,y)
\end{align*}
so $A(x,x) + A(y, y) = -2\mu\omega(x, y)$. The vectors $x$ and $y$ are not both $0$ (else $v$ would be 0), so $A(x,x)+A(y, y)>0$ and hence, $\omega(x, y)<0$.
\end{proof}

Let $\{\pm i\mu_1,\pm i\mu_2,\dots,\pm i\mu_n\}$ be a list of all eigenvalues of $\op{F}$, where $\mu_j>0$ and where each eigenvalue is repeated according to its multiplicity.

\begin{theorem}[Williamson \cite{Will}]\label{will}
There exists a symplectic basis $\{x_1, x_2, \dots, x_n, y_1, y_2, \dots, y_n\}$ of $V$ such that $\op{F}x_j=-\mu_j y_j$ and $\op{F}y_j=\mu_j x_j$. For $v=\sum_{j=1}^{n}(s_j x_j + t_j y_j) \in V$, we have $f(v)=\sum_{j=1}^{n}\mu_j (s_j^2 + t_j^2)$, where $\mu_j>0$.
\end{theorem}
\begin{proof}
Let $v\in V^{\C}$, $x, y\in V$ be as in Lemma \ref{lem3} and let $W=span\{x, y\}$. Then $W$ is an $\op{F}$-invariant symplectic plane in $V$ with symplectic basis $\{x_1=c x, y_1=c y\}$, where $c=\omega(y, x)^{-1/2}$. Therefore, the skew-orthogonal complement of $W$, 
$$W^\omega=\{v\in V\;|\; \omega(v, x)=\omega(v, y)=0\}$$
is an $\op{F}$-invariant subspace of $V$, and $V=W\oplus W^\omega$ of the symplectic subspaces $W$ and $W^\omega$. Moreover, since $A(v, w)=\omega (v, \op{F}w)$, we have $\omega(v, x)=\omega(v, y)=0$ iff $A(v, x)= A(v, y)=0$. Thus $W^\omega$ is equal to the $A$-orthogonal complement of $W$,
$$W^\omega=W^A=\{v\in V\;|\;A(v, x)=A(v, y)=0\}$$
By induction on the dimension of $V$, we may assume that Williamson's theorem is true for the $2n-2$ dimensional symplectic vector space $W^\omega$, which gives us a symplectic basis $\{x_2, x_3, \dots, x_n, y_2, y_3, \dots, y_n\}$ of $W^\omega$ such that $\op{F}x_j=-\mu_j y_j$ and $\op{F}y_j=\mu_j x_j$ for $j\geq 2$, where the number $\pm i\mu_j$ are the eigenvalues of the $\op{F}$ restricted to $W^\omega$. Then $\{x_1, x_2, \dots, x_n, y_1, y_2, \dots, y_n\}$ is a symplectic basis of $V$ such that $\op{F}x_j=-\mu_j y_j$ and $\op{F}y_j=\mu_j x_j$ for $j\geq 1$, where $\mu_1=\mu$.
For $v=\sum_{j=1}^{n}(s_j x_j +t_j y_j)\in V$, we obtain
\begin{align*}
f(v)&=A(v, v)= \omega(v, \op{F} v)\\
&=\sum_{j, k=1}^{n}\mu_k\omega (s_j x_j + t_j y_j, t_k x_k - s_k y_k)\\
&=\sum_{j=1}^{n}\mu_j (s_j^2+t_j^2)
\end{align*}
\end{proof}

Given an eigenvalue of $\op{F}$, $\lambda_j=i\mu_j$, the real $\mu_j>0$ is called a \textit{symplectic eigenvalue} of $\op{F}$ and $\{x_j, y_j\}$ is called the \textit{symplectic eigenvector pair} corresponding to $\mu_j$ and the symmetric bilinear form $A$ is called \textit{symplectically diagonalizable}.\\
The matrix representing $A$ in the symplectic basis given by Theorem \ref{will} is 
$$\begin{bmatrix}
    \op{D}_n & \op{0}_n\\
    \op{0}_n & \op{D}_n
\end{bmatrix}$$ where each $\op{D}_n=diag(\mu_1,\dots,\mu_n)$ is an $n\times n$ block diagonal matrix with $\mu_i>0$ being symplectic eigenvalues for all $i=1,2,\dots,n$ and $\op{0}_n$ is the $n\times n$ zero matrix.\\
If $B_{\lambda}= V_{\lambda}\oplus V_{\bar{\lambda}}$ is non-empty then $\mu$ is a symplectic eigenvalue of $\op{F}$ and $B_{\lambda}$ is the \textit{symplectic eigenspace} corresponding to the symplectic eigenvalue $\mu$. \\
One interpretation of Williamson's theorem is the following: 

If $\{\mu_1,\mu_2,\dots,\mu_k\}$ is a complete list of all symplectic eigenvalues of $\op{F}$ without repetition then $A$ is symplectically diagonalizable iff
\begin{equation}\label{1}V^{\C}=B_{\lambda_1}\oplus B_{\lambda_2}\oplus\cdots\oplus B_{\lambda_k}\end{equation} where each $B_{\lambda_j}$ is a symplectic subspace. In such a case, for every $z\in V^{\C}$ there exist unique vectors $z_i\in B_{\lambda_j}$ such that $z=z_1+z_2+\cdots+z_k$. This is called the \textit{symplectic eigenvector decomposition} of $z$.

One should note that the diagonalizing symplectic basis is not unique;
see Son \textit{et al.} \cite{son} for a detailed analysis of the set of diagonalizing bases.

The applications of Williamson's theorem are numerous, both in symplectic
geometry and topology, and in mathematical physics (Hamiltonian and quantum
mechanics); see for instance \cite{HZ} and \cite{Birk}. We mention that in a
recent preprint \cite{Kumar} Kumar and Tonny give an interesting account
of the developments of Williamson's result and its applications to operator
theory. \par One of the most important applications is the study of systems with
quadratic Hamiltonians, such as linear oscillators or small oscillations
around equilibrium points, because it helps in obtaining normal modes and
simplifies the analysis of the dynamics. In addition it allows, using the
theory of the metaplectic group the study of the quantum counterpart by
allowing the explicit calculations of the solutions of Schr\"{o}dinger's
equation for systems whose classical counterpart have a quadratic Hamiltonian.
This has immediate applications in the field of quantum optics because it
allows a detailed study of Gaussian beams and their propagation, permitting the decomposition of a multi-mode optical field into independent modes.

Suppose for instance the Hamiltonian $H$ is defined by
\[
H(z)=\frac{1}{2}\langle z,\operatorname{M}z\rangle \text{ \ , \ }\operatorname{M}>0\in\mathbb{M}(\R^{2n})
\]
the corresponding Hamilton equations $\dot{z}=\operatorname{J}\nabla_{z}H$ are the linear
system $\dot{z}=\operatorname{JM}z$ whose solution is simply $z(t)=e^{t\operatorname{JM}}z(0)$; due to the
fact that $\operatorname{JM}\in\operatorname*{\mathfrak{sp}}(\R^{2n})$ (the symplectic Lie algebra) we have
$\op{S}_{t}=e^{t\operatorname{JM}}\in\operatorname*{Sp}(\R^{2n})$ and the symplectic automorphisms
$\op{S}_{t}$ are explicitly calculated by considering the Williamson
diagonalization.

This reduces the study of the Hamiltonian system $\dot{z}=\operatorname{J}\nabla_{z}H$ \ to
that when $H$ has the simple form $H=\Lambda \langle x,x\rangle+\Lambda \langle p,p\rangle$ that
is in coordinates $z_{j}=(x_{j},p_{j})$,
\[
H(x,p)=\sum_{j=1}^{n}\lambda_{j}(x_{j}^{2}+p_{j}^{2})\text{ \ \ },\text{
\ \ }\lambda_{j}>0\text{.}%
\]
The solutions are well-know in this case ($H$ is a sum of harmonic
oscillators). 

\vspace{0.4cm}
Now, let $(V,\omega)$ be a $2n$-dimensional symplectic vector space equipped with a positive-semi definite quadratic form $f$.
To this quadratic form, corresponds a symmetric bilinear form $A$ on $V$ such that $f(v)=A(v,v)$ for all $v\in V$. There also exists a linear map $\op{F}:V\to V$ such that 
$$A(v,w)=\omega(v,\op{F}w)$$ for all $v,w\in V$.

Let $V_{(\lambda)}\subseteq V^\C$ denote the generalized eigenspace of $\op{F}$ corresponding to eigenvalue $\lambda\in\C$.
We also define the \textit{radical} of the bilinear form $A$ as:
$$Rad (A) = \{x\in V^\C\;|\;A(x,y)=0\;;\;\forall y\in V^\C\}=Ker(\op{F})\subset V_{(0)}$$.

\begin{lemma}\label{lem9}
Let $v\in V_{(\lambda)}$ and $w\in V_{(\mu)}$. If $\lambda+\mu\neq 0$, then $\omega(v,w)=0$.
\end{lemma}
\begin{proof}
When $\lambda+\mu\neq 0$,
notice that $V_{(\lambda)}$ is invariant under $\op{F}+\mu\op{I}$. Moreover, $\op{F}+\mu\op{I}$ is injective as a map $V_{(\lambda)}\to V_{(\lambda)}$, making it and its powers bijections since the space is finite dimensional. Therefore, for $N$ large enough, there exists a unique $u\in V_{(\lambda)}$ st. $v=(\op{F}+\mu\op{I})^N u\in V_{(\lambda)}$. Also take $w\in V_{(\mu)}$ respectively, then we have:
$$\omega(v,w)=\omega((\op{F}+\mu\op{I})^Nu,w)=\omega(u,(-\op{F}+\mu\op{I})^N w)=0$$
\end{proof}
\begin{proposition}\label{prop1}
Given a positive-semi definite symmetric bilinear form $A$ on a vector space $V$, if $A(v,v)=0$ for $v\in V$ then $A(v,w)=0$ for all $w\in V$ i.e. $v\in Ker(\op{F})$.     
\end{proposition}
\begin{proof}
$$0\leq A(v+tw, v+tw)= A(v,v)+2t A(v,w) +t^2 A(w,w)$$
The means that the discriminant of this quadratic must be non-positive.
$$\Rightarrow 4A(v,w)^2\leq 0 \Rightarrow A(v,w)=0$$
\end{proof}

\begin{lemma}\label{lem7}
All non-zero eigenvalues of $\op{F}$ are purely imaginary.
\end{lemma}
\begin{proof}
Suppose there existed $\lambda\neq 0\in\C$ with $\Re \lambda\neq 0$ and $V_{(\lambda)}\neq \{0\}$. For $v\in V^{\C}$, set $\Re v=\frac{1}{2}(v+\overline{v})\in V$. Then 
$$W_{(\lambda)}=\Re(V_{(\lambda)})=\{\Re v\;|\;v\in V_{(\lambda)}\}$$
is a linear subspace of $V$. We claim that \begin{enumerate}\item\label{i1}$W_{(\lambda)}\neq \{0\};$\item\label{i2}$A(x,y)=0$ for all $x,y\in W_{(\lambda)}$ \end{enumerate}
Assume that $\Re v=0$ for all $v\in V_{(\lambda)}$; then $v+\overline{v}=0$ for all $v\in V_{(\lambda)}$. But if $v\in V_{(\lambda)}$, then $iv\in V_{(\lambda)}$, so $iv+\overline{iv}=v-\overline{v}=0$ and hence $v=0$. This contradicts the assumption $V_{(\lambda)}\neq \{0\}$, thus proving \ref{i1}. Now let $x, \op{F}y\in W_{(\lambda)}$. Writing $x=\Re u$ and $\op{F}y=\Re v$ with $u, v\in V_{(\lambda)}$, we find
\begin{align*}
A(x,y)=\omega(x,\op{F}y)&=\frac{1}{4}\omega(u+\overline{u},v+\overline{v})\\
&=\frac{1}{4}\omega(u, v)+\frac{1}{4}\omega(u,\overline{v})+\frac{1}{4}\omega(\overline{u}, v)+\frac{1}{4}\omega(\overline{u},\overline{v})
\end{align*}
We have $\overline{u}, \overline{v}\in V_{(\bar{\lambda})}$ because $\op{F}$ is a real linear map. Since $\Re \lambda\neq 0$, we have $\lambda+\bar{\lambda}\neq 0$, so 
$\omega (u, \overline{v})=\omega(\overline{u}, v)=0$ by Lemma \ref{lem5}. Similarly, $\omega(u, v)=\omega(\overline{u}, \overline{v})=0$ by Lemma \ref{lem5}, because $\lambda \neq 0$. Thus, $A(x, y)=0$, proving \ref{i2}. The assertions $\ref{i1}-\ref{i2}$ imply $V_{(\lambda)}\subseteq Ker(\op{F})\subseteq V_{(0)}$ due to Proposition \ref{prop1}. This contradicts $\lambda\neq 0$.\\
Therefore, $\Re\lambda=0$.
\end{proof}
\begin{lemma}\label{lem6}
For an operator $\op{F}$ on a vector space $V$ satisfying $\omega(v,\op{F}w)=-\omega(\op{F}v,w)$ for $v,w\in V$, we have $Im(\op{F})=(Ker(\op{F}))^\omega$.   
\end{lemma}
\begin{proof}
Take $k,w\in Ker(\op{F}),Im(\op{F})$ respectively. Then $w=\op{F}v$ for some $v\in V$ and
$$\omega(k,\op{F}v)=-\omega(\op{F}k,v)=0$$
Therefore, $Im(\op{F})\subseteq (Ker(\op{F}))^\omega$.\\
Finally, we note from the rank-nullity and symplectic dimension theorem that $dim (Im(\op{F}))=dim(V)-dim(Ker(\op{F}))=dim((Ker(\op{F}))^\omega)$.\\
This concludes the proof.
\end{proof}
\begin{lemma}\label{lem8}
$V_{(\lambda)}=V_{\lambda}$ for $\lambda\neq 0$.
\end{lemma}
\begin{proof}
Let us first consider the Jordan decomposition of $V^{\C}$ into the generalized eigenspaces of $\op{F}$ i.e.
$$V^{\C}=V_{(0)}\bigoplus_{\lambda\neq 0} V_{(\lambda)}$$
Let $U=(\bigoplus_{\lambda\neq 0} V_{(\lambda))})\cap V$, notice that $U^{\C}=\bigoplus_{\lambda\neq 0} V_{(\lambda)}$.\\
We know that $U$ is $\op{F}$-invariant and we find that $A$ is positive-definite on $U$ because\\ $Rad(A)\cap U=\{0\}$.

Finally, using Theorem \ref{will} on $\op{F}|_{U}$, the statement of the lemma follows.
\end{proof}
\begin{theorem}[Hörmander \cite{Horm}]\label{horm1}
There exists a symplectic basis \\$\{x_1,x_2,\dots,x_n,y_1,y_2,\dots,y_n\}$ of $V$ such that for $v=\sum_{j=1}^{n}(s_jx_j+t_jy_j)\in V$, we have \\$f(v)=\sum_{j=1}^{k}\mu_j (s_j^2+t_j^2)+\sum_{j=k+1}^{k+l}s_j^2$, for $1\leq k\leq k+l\leq n$, where $\mu_j>0$.
\end{theorem}
\begin{proof}
First of all, we know that $V^\C$ can be Jordan-decomposed into the generalized eigenspaces of $\op{F}$. Namely,
$$V^\C=Y\oplus V_{(0)}$$
where $Y=\bigoplus_{j=1}^{k} (V_{(\lambda_j)}\oplus V_{(\bar{\lambda}_j)})$.
Applying Lemma \ref{lem8} to each generalized eigenspace in $Y$, we have that  $V_{(\lambda_j)}=V_{\lambda_j}$ for $\lambda_j\neq 0$, where each $\lambda_j$ is purely imaginary due to Lemma \ref{lem7}.

Now take $v\in V_{(\lambda)}$. 
If $\lambda=i\mu\neq 0$ then consider $x,y\in V$ as in Lemma \ref{lem3} and let $W=span\{x,y\}$. $W$, is an $\op{F}$-invariant symplectic plane in $V$. We normalize $x$ \& $y$ such that $\omega(x,y)=1$  and repeat this procedure on $W^\omega\subset V$.

If $v\in V_{(0)}$, then take $y=\Re v$. We will consider three exhaustive cases.

Case 1: Suppose $\op{F^2}y=0$, $A(y,y)=\omega(y,\op{F}y)\neq 0$.\\
Then $\op{F}y\neq 0$, so $y\notin Ker(\op{F})$. $W=span\{y,\op{F}y\}$, is an $\op{F}$-invariant symplectic plane in $V$. We normalize $y$ \& $\op{F}y$ such that $\omega(y,\op{F}y)=1$.\\
If $w=t_1 y+t_2 \op{F}y\in W$ then $$A(w,w)=\omega(w,\op{F}w)=t_1^2$$
And now, we repeat this process on $W^\omega\subset V$.

Case 2: Suppose we can pick two real vectors $v,w\in Ker(\op{F})$ such that $\omega(v,w)\neq 0$.\\
Then $W=span\{v,w\}$, is an $\op{F}$-invariant symplectic plane in $V$. We normalize $v$ \& $w$ such that $\omega(v,w)=1$.\\
If $z=t_1 v + t_2 w\in W$ then $$A(z,z)=0$$
And now, we repeat this process on $W^\omega\subset V$.

Case 3: \begin{enumerate}[label=(\roman*)]
     \item $Ker(\op{F})$ is isotropic;
    \item If $v\in Ker(\op{F^2})$ then $A(v,v)=0$
\end{enumerate}
From (ii), we can deduce that $v\in Rad(A)=Ker(\op{F})$ using Proposition \ref{prop1}. \\Thus, $Ker(\op{F^2})=Ker(\op{F})$.

Now consider the short exact sequence: $$0\to Ker(\op{F})\xhookrightarrow{}Ker(\op{F^2})\to Ker(\op{F})\cap Im(\op{F})\to 0$$ 
From Lemma \ref{lem6}, we know that $Im(\op{F})=(Ker(\op{F}))^\omega$, therefore, \\$Ker(\op{F})\cap Im(\op{F})=Ker(\op{F})\cap(Ker(\op{F}))^\omega=Ker(\op{F})$, from assumption (i).\\
Finally, we can conclude that $$Ker(\op{F})\cong Ker(\op{F^2})/Ker(\op{F})=\{0\}$$
Therefore, $V_{(0)}=\{0\}$. This makes this case trivial

Finally, after exhausting the dimension, we get a symplectic basis \\$\{x_1,x_2,\dots,x_n,y_1,y_2,\dots,y_n\}$ of $V$ such that for $v=\sum_{j=1}^{n}(s_jx_j+t_jy_j)\in V$, we have $$f(v)=\sum_{j=1}^{k}\mu_j (s_j^2+t_j^2)+\sum_{j=k+1}^{k+l}s_j^2$$
\end{proof}

The matrix representing $A$ in the symplectic basis given by Theorem \ref{horm1} is 
$$\begin{bmatrix}
    \op{A}_n & \op{0}_n\\
    \op{0}_n & \op{B}_n
\end{bmatrix}$$
where $\op{A}_n= diag(\mu_1,\dots,\mu_k,\underbrace{1,\dots,1}_{l},\underbrace{0,\dots,0}_{n-k-l})$ and $\op{B}_n= diag(\mu_1,\dots,\mu_k,\underbrace{0,\dots,0}_{n-k})$ \vspace{0.3cm}\\are $n\times n$ block matrices where $\mu_i>0$ for $i=1,2,\dots,k$ are the symplectic eigenvalues of the non-degenerate part of $\op{F}$ and $\op{0}_n$ is the $n\times n$ zero matrix.

There is a special case of Theorem \ref{horm1} that will be of importance to us.

\begin{corollary}\label{semi}
If $Ker(\op{F})$ is a symplectic subspace, then there exists a symplectic basis \\$\{x_1,x_2,\dots,x_n,y_1,y_2,\dots,y_n\}$ of $V$ such that for $v=\sum_{j=1}^{n}(s_jx_j+t_jy_j)\in V$, we have \\$f(v)=\sum_{j=1}^{k}\mu_j (s_j^2+t_j^2)$.
\end{corollary}
\begin{proof}
From Theorem \ref{horm2}, there exists a symplectic basis $\{x_1,x_2,\dots,x_n,y_1,y_2,\dots,y_n\}$ of $V$ such that for $v=\sum_{j=1}^{n}(s_jx_j+t_jy_j)\in V$, we have $$f(v)=\sum_{j=1}^{k}\mu_j (s_j^2+t_j^2)+\sum_{j=k+1}^{k+l}s_j^2$$
The above formula show that $f$ vanishes on $U=span\{x_{k+l+1},\dots,x_n,y_{k+1},\dots,y_n\}$, implying that $Rad(A)=Ker(\op{F})=U$. The subspace $U$ is symplectic if and only if $l=0$.\\
Thus, we get a symplectic basis $\{x_1,x_2,\dots,x_n,y_1,y_2,\dots,y_n\}$ of $V$ such that \\for $v=\sum_{j=1}^{n}(s_jx_j+t_jy_j)\in V$, we have $$f(v)=\sum_{j=1}^{k}\mu_j (s_j^2+t_j^2)$$
\end{proof}

The matrix representing $A$ in the symplectic basis given by Corollary \ref{semi} is 
$$\begin{bmatrix}
    \op{D}_n & \op{0}_n\\
    \op{0}_n & \op{D}_n
\end{bmatrix}$$
where $\op{D}_n= diag(\mu_1,\dots,\mu_k,\underbrace{0,\dots,0}_{n-k})$ \vspace{0.3cm}\\is an $n\times n$ block matrices where $\mu_i>0$ for $i=1,2,\dots,k$ are the symplectic eigenvalues of the non-degenerate part of $\op{F}$ and $\op{0}_n$ is the $n\times n$ zero matrix.

See the Appendix (Section \ref{appen}) for the proof of Corollary \ref{semi} when $V=\R^{2n}$.

\begin{definition}\label{def1}
    Let $A$ be a symmetric bilinear form on $V$ with signature $(p,q)$, then $p$ is the maximal dimension of the subspace where $A$ is positive-definite and likewise, $q$ is the maximal dimension of the subspace where $A$ is negative-definite.
\end{definition}
From \textit{Sylvester's law of inertia}, we know that the signature of a symmetric bilinear form is an invariant.

Let us now consider the case when $A$ is hyperbolic i.e. has signature $(p,1)$.\\
We can no longer guarantee that Proposition \ref{prop1} holds true. We can see why this is in the example below:

\begin{example}\label{ex1}
$V=\R^2$, $A=\begin{bmatrix}0&1\\1&0\end{bmatrix}$ i.e. $A\left(\begin{bmatrix} x_1\\x_2
\end{bmatrix},\begin{bmatrix}y_1\\y_2\end{bmatrix}\right)=x_2y_1+x_1y_2$.

$A(e_1,e_1)=A(e_2,e_2)=0$ and $A(e_1,e_2)=1$, where $e_j$ denotes the $j$-th standard basis vector.\\
Also, $A(\frac{1}{2}(e_1+e_2,e_1+e_2))=1$ and $A(\frac{1}{2}(e_1-e_2,e_1-e_2))=-1$.

Therefore, $A$ is positive-definite on $span\{v_1=\frac{1}{\sqrt{2}}(e_1+e_2)\}$ and negative-definite on \\$span\{v_2=\frac{1}{\sqrt{2}}(e_1-e_2)\}$.

Note that while $Rad(A)=\{0\}$, $A$ vanishes on the one dimensional subspaces $W_1=span\{e_1\}$ and $W_2=span\{e_2\}$ that are disjoint from the radical. \\We call such subspaces $A$-isotropic subspaces and $H_2=span\{v_1,v_2\}$ the two dimensional hyperbolic plane. 
\end{example}

\begin{claim}\label{claim1}
For a symmetric bilinear form $A$ with signature $(p,1)$, all the $A$-isotropic subspaces are one dimensional.
\end{claim}
\begin{proof}
We can decompose $V$ into an $A$-orthogonal direct sum: $$V=Rad(A)\oplus V_{+}\oplus V_{-}$$
where $A$ is positive-definite on the $p$ dimensional subspace $V_{+}$ and likewise negative-definite on the one dimensional subspace $V_{-}$\\
Now choose vectors $v_{+},v_{-}\neq 0\in V_{+},V_{-}$ respectively such that $H=span\{v_{+},v_{-}\}$. With the appropriate normalization, we have $H\cong H_2$ from comparison to Example \ref{ex1}.\\ Therefore, we have the one dimensional $A$-isotropic subspaces $W_1=span\{v_{+}+v_{-}\}$ and $W_2=span\{v_{+}-v_{-}\}$.\\
We have no more $A$-isotropic subspaces since $A$ is positive-semi definite on $H^A$, the $A$-complement of $H$ and the statement of the claim follows.
\end{proof}

\begin{claim}\label{claim2}
$A$ cannot be negative-semi definite on a two dimensional subspace that has trivial intersection with $Rad(A)$.
\end{claim}
\begin{proof}
The only ways to possibly construct a two dimensional subspace where $A$ is negative-semi definite is by taking the span of any two vectors from the set $\{v_{-},v_{+}+v_{-},v_{+}-v_{-}\}$ (defined in Claim \ref{claim1}). It is obvious that $A\geq 0$ on $span\{v_{+}+v_{-},v_{-}-v_{-}\}$.\\
Set $W_{\pm}=span\{v_{-},v_{+}\pm v_{-}\}$, then we see that $v_{+}\in W_{\pm}$. Thus, $A$ is not negative-semi definite on $W_{\pm}$.

Therefore, all our attempts to construct such a subspace fail and the statement of the claim follows.
\end{proof}

\begin{claim}\label{claim3}
If $\lambda-\bar{\lambda}\neq 0$ and $W_{(\lambda)}=\Re(V_{(\lambda)})$ then $dim_{\R}W_{(\lambda)}$ has even dimension.
\end{claim}
\begin{proof}
Assume $\lambda-\bar{\lambda}\neq 0$, we know that \begin{align*}W_{(\lambda)}=\Re(V_{(\lambda)})&=\{v+\bar{v}\;|\;v\in V_{(\lambda)}\}\\
\Im(V_{(\lambda)})&=\{-i(v-\bar{v})\;|\;v\in V_{(\lambda)}\}\end{align*}
We will first show that $W_{(\lambda)}\oplus i\Im(V_{(\lambda)})=V_{(\lambda)}\oplus V_{(\bar{\lambda})}$.

It is true by definition that $W_{(\lambda)}\oplus i\Im(V_{(\lambda)})\subseteq V_{(\lambda)}\oplus V_{(\bar{\lambda})}$.\\
Now if $v\in V_{(\lambda)}$, then write $v=a+ib$ and $\bar{v}=a-ib$, \\so $a=\frac{1}{2}(v+\bar{v})\in W_{(\lambda)}$ and $b=\frac{1}{2i}(v-\bar{v})\in\Im(V_{(\lambda)})$.\\ 
Therefore, $v=a+ib\in W_{(\lambda)}\oplus i\Im(V_{(\lambda)})\Rightarrow V_{(\lambda)}\oplus V_{(\bar{\lambda})}\subseteq W_{(\lambda)}\oplus i\Im(V_{(\lambda)})$.

Now will show that $\Im(V_{(\lambda)})=W_{(\lambda)}$.

Let $x\in W_{(\lambda)}$ i.e. $x=\Re v$ with $v\in V_{(\lambda)}$. \\Then $w=iv\in V_{(\lambda)}$ and $\Im w=\frac{1}{2i}(w-\bar{w})=\frac{1}{2}(v+\bar{v})=\Re v=x$. \\Therefore, $W_{(\lambda)}\subseteq \Im(V_{(\lambda)})$ and we establish equality since both subspaces have the same dimension.

Finally, we have $W_{(\lambda)}\oplus iW_{(\lambda)}=V_{(\lambda)}\oplus V_{(\bar{\lambda})}$.
Since $\lambda+\bar{\lambda}\neq 0$, $2\;dim_{\R}W_{(\lambda)}=2\;dim_{\R}V_{(\lambda)}$.\\
Therefore, $dim_{\R}W_{(\lambda)}=2\;dim_{\C}V_{(\lambda)}$ is even.
\end{proof}

\begin{theorem}[Hörmander \cite{Horm}]\label{horm2}
There exists a symplectic basis \\$\{x_1,x_2,\dots,x_n,y_1,y_2,\dots,y_n\}$ of $V$ such that for $v=\sum_{j=1}^{n}(s_jx_j+t_jy_j)\in V$, we have \\$f(v)=\sum_{j=1}^{k}\mu_j (s_j^2+t_j^2)+\sum_{j=k+1}^{k+l}s_j^2+q(s,t)$,
where $\mu_j>0$\\ and either $k+l<n-1$ and $q(s,t)=s_n^2-2t_{n-1}t_n$ \\or $k+l<n$ and $q(s,t)=-s_n^2$ or $q(s,t)=2\lambda s_n t_n$, $\lambda>0$.
\end{theorem}
\begin{proof}
First of all, we know that $V^\C$ can be Jordan-decomposed into the generalized eigenspaces of $\op{F}$. Namely,
$$V^\C=Y\oplus V_{(0)}$$
where $Y=\bigoplus_{j=1}^{k} (V_{(\lambda_j)}\oplus V_{(\bar{\lambda}_j)})$.\\
Take $v\in V_{(\lambda)}$. 

If $\Re\lambda\neq 0$, then it follows from Lemma \ref{lem7} that $A$ vanishes on $W_{(\lambda)}$ and since $W_{(\lambda)}\subseteq V_{(\lambda)}\cap Ker(\op{F})=\{0\}$, we know from Claim \ref{claim1} that $W_{(\lambda)}$ is a one dimensional $A$-isotropic subspace. Using Claim \ref{claim3}, we must have $\lambda=\bar{\lambda}$, otherwise $dim_{\C}V_{(\lambda)}=\frac{1}{2}$ which is not possible.\\
Therefore, $dim_{\C}V_{(\lambda)}=1$ and $\lambda\in\R$.\\
Now choose a real vector $x\in V_{(\lambda)}$ and a real vector $y\in V_{(\bar{\lambda})}$ such that $\omega(x,y)=1$. Then $W=span\{x,y\}$ is an $\op{F}$-invariant symplectic plane and we have $\op{F}x=\lambda x$, $\op{F}y=-\lambda y$. If $z=s_1 x + t_1 y$, we have:
$$A(z,z)=2\lambda s_1 t_1$$
And now, we repeat this procedure on $W^\omega\subset V$.

If $\lambda=i\mu$, $\mu>0$ and $\op{F}v=i\mu v$ then take $x,y\in V$ such that $v=x+iy\in V_{(\lambda)}$. We have $\op{F}x=-\mu y$ and $\op{F}y=\mu x$.\\
Using Claim \ref{claim2}, $A$ cannot be negative-semi definite on the two dimensional subspace spanned by $\{x,y\}$ that intersects trivially with $Ker(\op{F})$.\\
Thus $A(x,x)=-\mu\omega(x,y)> 0$ and $\omega(x,y)<0$.
$W=span\{x,y\}$ is an $\op{F}$-invariant symplectic plane. We normalize $x$ \& $y$ such that $\omega(x,y)=1$. If $z=s_1 x + t_1 y$ then:
$$A(z,z)=\mu (s_1^2+t_1^2)$$
And now, we repeat this procedure on $W^\omega\subset V$.

If $v\in V_{(0)}$, then take $y=\Re v$. We will consider three exhaustive cases.

Case 1: Suppose $\op{F^2}y=0$, $A(y,y)=\omega(y,\op{F}y)\neq 0$.\\
Then $\op{F}y\neq 0$, so $y\notin Ker(\op{F})$. And $W=span\{y,\op{F}y\}$, with the appropriate normalization, is an $\op{F}$-invariant symplectic plane in $V$.\\
If $w=t_1 y+t_2 \op{F}y\in W$ then: $$A(w,w)=\omega(w,\op{F}w)=t_1^2$$
And now, we repeat this process on $W^\omega\subset V$.

Case 2: Suppose we can pick two real vectors $v,w\in Ker(\op{F})$ such that $\omega(v,w)\neq 0$.\\
Then $W=span\{v,w\}$, with the appropriate normalization, is an $\op{F}$-invariant symplectic plane in $V$.\\
If $z=t_1 v + t_2 w\in W$ then: $$A(z,z)=0$$
And now, we repeat this process on $W^\omega\subset V$.

Case 3: \begin{enumerate}[label=(\roman*)]
     \item $Ker(\op{F})$ is isotropic;
    \item If $v\in Ker(\op{F^2})$ then $A(v,v)=0$
\end{enumerate}
From Claim \ref{claim2}, we find that $Ker(\op{F^2})/Ker(\op{F})$ is one dimensional.\\
Now consider the short exact sequence: $$0\to Ker(\op{F})\xhookrightarrow{}Ker(\op{F^2})\to Ker(\op{F})\cap Im(\op{F})\to 0$$ 
From Lemma \ref{lem6}, we know that $Im(\op{F})=(Ker(\op{F}))^\omega$, therefore, \\$Ker(\op{F})\cap Im(\op{F})=Ker(\op{F})\cap(Ker(\op{F}))^\omega=Ker(\op{F})$, from assumption (i).\\
We now have $Ker(\op{F^2})/Ker(\op{F})\cong Ker(\op{F})$, therefore $Ker(\op{F})$ has dimension one.

It now follows from the Jordan canonical form of $\op{F}$ that $\Re(V_{(0)})$ has the basis \\$\{x,\op{F}x,\dots,\op{F^{N-1}}x\}$ while $\op{F^N}x=0$.\\
Since $V_{(0)}$ is a symplectic subspace, $N$ must be even. 

However, if $N=2$ then the basis is $\{x,\op{F}x\}$ such that $x\in Ker(\op{F^2})$, implying $A(x,x)=\omega(x,\op{F}x)=0$ from assumption (ii). This is a contradiction since $V_{(0)}$ is a symplectic subspace.\\
Therefore, we deduce that $N>2$ must be even.

We know that $A(\op{F^j}x,\op{F^k}x)=\omega(\op{F^j}x,\op{F^{k+1}}x)=0$ if $j+k+1\geq N$.\\
Then: \\$A(\op{F^{N-3}}x,\op{F^{N-2}}x)=0$ if $N\geq 4$; \\$A(\op{F^{N-3}}x,\op{F^{N-3}}x)=0$ if $N\geq 5$ and\\ $A(\op{F^{N-2}}x,\op{F^{N-2}}x)=0$ if $N\geq 3$.\\
This means that $A$ vanishes on $span\{\op{F^{N-3}}x,\op{F^{N-2}}x\}$ if $N\geq 5$.\\ We have $span\{\op{F^{N-3}}x,\op{F^{N-2}}x\}\cap Ker(\op{F})=\{0\}$ since $span\{\op{F^{N-1}}x\}=Ker(\op{F})$. 

Therefore, we have found a two dimensional subspace where $A$ vanishes which intersects trivially with $Rad(A)$. Due to Claim \ref{claim2}, we know that this is not possible.\\Hence it must be the case that $2<N<5$. $N$ being even, the only possibility is $N=4$.

We must also have $\omega(x,\op{F^3}x)<0$.\\
Suppose, for the sake of contradiction, that $\omega(x,\op{F^3}x)\geq 0$. Then \\
$\omega(\op{F}x,\op{F^2}x)=-\omega(x,\op{F^3}x)\leq 0$.\\
Let $W=span\{\op{F}x,\op{F^2}x\}$. \\Then $W\cap Ker(\op{F})=\{0\}$ and also $A\leq 0$ on $W$:\\
$A(\op{F}x,\op{F^2}x)=\omega(\op{F}x,\op{F^3}x)=0$\\
$A(\op{F}x,\op{F}x)=\omega(\op{F}x,\op{F^2}x)\leq 0$\\
$A(\op{F^2}x,\op{F^2}x)=\omega(\op{F^2}x,\op{F^3}x)=0$\\
This is a contradiction due to Claim \ref{claim3}.

We normalize so that $\omega(x,\op{F^3}x)=-1$ and set $y=x+t\op{F^2}x$.\\
Then $\omega(y,\op{F}y)=\omega(x,\op{F}x)-2t=0$ if $t=\frac{1}{2}\omega(x,\op{F}x)$.\\
The $\op{F}$-invariant symplectic subspace $S=Re(V_{(0)})$ is thus spanned by the symplectic basis:\\
$\{e_1=-\op{F^3}y,e_2=-\op{F}y,f_1=y,f_2=\op{F^2}y\}$.\\
If $z=s_1 e_1+s_2 e_2+t_1 f_1+t_2 f_2$, then:
$$A(z,z)=s_2^2-2t_1 t_2$$
And now, we repeat this procedure on $S^\omega\subset V$.

Finally, after exhausting the dimension of $V$, we get a symplectic basis \\$\{x_1,x_2,\dots,x_n,y_1,y_2,\dots,y_n\}$ of $V$ such that for $v=\sum_{j=1}^{n}(s_jx_j+t_jy_j)\in V$, we have $$f(v)=\sum_{j=1}^{k}\mu_j (s_j^2+t_j^2)+\sum_{j=k+1}^{k+l}s_j^2+q(s,t),\;
\text{where}\; \mu_j>0$$ and either $k+l<n-1$ and $$q(s,t)=s_n^2-2t_{n-1}t_n$$ or $k+l<n$ and $$q(s,t)=-s_n^2\;\text{or} \;q(s,t)=2\lambda s_n t_n,\;\text{where}\;\lambda>0.$$
\end{proof}

The matrix representing $A$ in the symplectic basis given by Theorem \ref{horm2} is 

$$\begin{bmatrix}
    \op{A}_n & \op{C}_n\\
    \op{C}_n & \op{B}_n
\end{bmatrix}$$
Now, if $k+l<n-1$ and $q(s,t)=s_n^2-2t_{n-1}t_n$ then

$\op{A}_n= diag(\mu_1,\dots,\mu_k,\underbrace{1,\dots,1}_{l},\underbrace{0,\dots,0}_{n-k-l-1},1)$ and $\op{B}_n=\begin{bmatrix}
\mu_1& & & & & & & \\
& &\mu_2& & & & & & \\
& & & \ddots& & & & & \\
& & & &\mu_k& & & & \\
& & & & &0& & & \\
& & & & & &\ddots& &\\
& & & & & & &0& -1\\
& & & & & & & -1& 0
\end{bmatrix}$ \vspace{0.3cm}\\ are $n\times n$ block matrices where $\mu_i>0$ for $i=1,2,\dots,k$ are the symplectic eigenvalues of the non-degenerate part of $\op{F}$, the blank entries in $\op{B}_n$ are all zero and $\op{C}_n=\op{0}_n$ is the $n\times n$ zero matrix.

If $k+l<n$ and $q(s,t)=-s_n^2$ then

$\op{A}_n= diag(\mu_1,\dots,\mu_k,\underbrace{1,\dots,1}_{l},\underbrace{0,\dots,0}_{n-k-l-1},-1)$ and $\op{B}_n= diag(\mu_1,\dots,\mu_k,\underbrace{0,\dots,0}_{n-k})$
\vspace{0.3cm}\\ are $n\times n$ block matrices where $\mu_i>0$ for $i=1,2,\dots,k$ are the symplectic eigenvalues of the non-degenerate part of $\op{F}$ and $\op{C}_n=\op{0}_n$ is the $n\times n$ zero matrix.

If $k+l<n$ and $q(s,t)=2\lambda s_n t_n$, $\lambda>0$ then

$\op{A}_n= diag(\mu_1,\dots,\mu_k,\underbrace{1,\dots,1}_{l},\underbrace{0,\dots,0}_{n-k-l})$ and $\op{B}_n= diag(\mu_1,\dots,\mu_k,\underbrace{0,\dots,0}_{n-k})$\vspace{0.3cm}\\ are $n\times n$ block matrices where $\mu_i>0$ for $i=1,2,\dots,k$ are the symplectic eigenvalues of the non-degenerate part of $\op{F}$ and $\op{C}_n=diag(\underbrace{0,\dots,0}_{n-1},\lambda)$ is an $n\times n$ block matrix.

\section{Results}


\begin{proposition}\label{lem4}
Given a positive-definite quadratic form $f$ on $V$, let $P$ be a subspace of $V^{\C}$ then $P$ is $\op{F}$-invariant iff \begin{equation}\label{eq:2}
P=(B_{\lambda_1}\cap P)\oplus (B_{\lambda_2}\cap P)\oplus\cdots\oplus (B_{\lambda_k}\cap P)
\end{equation}
\end{proposition}
\begin{proof}
Assume the decomposition in (\ref{eq:2}) holds. \\Let $u\in P\subseteq V^{\C}\neq 0$ then
$u=u_1+u_2+\cdots+u_k$ with $u_j\in B_{\lambda_j}\cap P$. \\Take $u_j=u_{j_1}+u_{j_2}$ where $u_{j_1}\in E_{\lambda_j}\cap P$ and $u_{j_2}\in E_{\bar{\lambda}_j}\cap P$.\\
We have $\op{F}u_{j_1}=i\mu_j u_{j_1}\in P$ and $\op{F}u_{j_2}=-i\mu_j u_{j_2}\in P$. Therefore, $u_j\in B_{\lambda_j}\cap P$.\\
Conversely, assume that $P$ is $\op{F}$-invariant. Let $u\in P$. We must show that in $u=u_1+u_2+\cdots+u_k$ each $u_j=u_{j_1}+u_{j_2}\in E_{\lambda_j}\oplus E_{\bar{\lambda}_j}=B_{\lambda_j}$ is contained in $P$. \\
Let $p_1,\dots,p_k,p_{k+1},p_{2k}$ be the Lagrange interpolating polynomials for $i\mu_1,\dots,i\mu_k,-i\mu_1,-i\mu_k\in\C$. Then $p_j(i\mu_l)=\delta_{jl}$\;;\;$p_{k+j}(-i\mu_l)=\delta_{jl}$ and  $p_{k+j}(i\mu_l)=p_{j}(-i\mu_l)=0$ for all $1\leq j,l\leq k$. \\
So, for $1\leq j\leq k$
$$p_j(\op{F})u=p_j(\op{F})(u_{1_1}+u_{1_2}+\dots+u_{k_1}+u_{k_2})=u_{j_1}$$
$$p_{k+j}(\op{F})u=p_{k+j}(\op{F})(u_{1_1}+u_{1_2}+\dots+u_{k_1}+u_{k_2})=u_{j_2}$$
Since $P$ is $\op{F}$-invariant, $P$ is $p_j(\op{F})$-invariant for all $1\leq j\leq 2k$, so this shows that \\
$u_j=u_{j_1}+u_{j_2}=(p_j(\op{F})+p_{k+j}(\op{F}))u\in P$.
\end{proof}

Now suppose that $V$ is equipped with a second positive-definite quadratic form $g$. Let $B$ be the symmetric bilinear form corresponding to $g$ and $\op{G}$ the linear map defined by $B(v, w)=\omega (v, \op{G}w)$.
\begin{lemma}\label{lem5}
The maps $\op{F}$ and $\op{G}$ commute i.e. $[\op{F}, \op{G}]=0$ iff the functions $f$ and $g$ Poisson commute i.e. $\{f, g\}=0$.
\end{lemma}
\begin{proof}
The linear maps $\op{F}$ and $\op{G}$, viewed as vector fields on $V$, up to a factor of 2, are nothing but the Hamiltonian vector fields of the quadratic functions $f$ and $g$ respectively. Using the fact that two functions \textit{Poisson} commute iff their Hamiltonian vector fields commute, we finish the proof.
\end{proof}
\begin{rem} On the standard symplectic vector space $V=\R^{2n}$, we have $\op{F}=\op{JM_A}$ and $\op{G}=\op{JM_B}$, so $\{f, g\}=0$ iff $[\op{JM_A}, \op{JM_B}]=0$.
\end{rem}
A standard result from linear algebra says that two diagonalizable maps are simultaneously diagonalizable iff they commute. The maps $\op{F}$ and $\op{G}$ are diagonalizable by Williamson's theorem, so from Lemma \ref{lem5}, we deduce that $\{f, g\}=0$ iff $\op{F}$ and $\op{G}$ are simultaneously diagonalizable. In fact, something stronger is true.
\begin{theorem}\label{simul}
Let $f$ and $g$ be positive-definite quadratic forms on $V$. Then $f$ and $g$ Poisson commute iff there exists a symplectic basis $\{x_1, x_2, \dots, x_n, y_1, y_2, \dots, y_n\}$ of $V$ such that 
\begin{align*}
\op{F}x_j&= -\mu_j y_j \;\;\text{and}\;\; \op{F}y_j= \mu_j x_j,\\
\op{G}x_j&= -\gamma_j y_j \;\;\text{and}\;\; \op{G}y_j= \gamma_j x_j
\end{align*}
Here $\{\pm i\mu_1, \pm i\mu_2, \dots, \pm i\mu_n\}$ and $\{\pm i\gamma_1, \pm i\gamma_2, \dots, \pm i\gamma_n\}$ is a list of all eigenvalues of $\op{F}$ and $\op{G}$ respectively, where $\mu_j, \gamma_j>0$ and where each eigenvalue is repeated according to its multiplicity. For $v=\sum_{j=1}^{n}(s_j x_j +t_j y_j)\in V$, we have
$$f(v)=\sum_{j=1}^{n}\mu_j (s_j^2+t_j^2), \;\;\; g(v)=\sum_{j=1}^{n}\gamma_j (s_j^2+t_j^2)$$
\end{theorem}
\begin{proof}
Suppose there exists such a basis . The $\op{F}$ and $\op{G}$ commute, so $f$ and $g$ Poisson commute by Lemma \ref{lem5}. Now suppose that $f$ and $g$ Poisson commute, then $\op{F}$ and $\op{G}$ commute by Lemma \ref{lem5}. So, there exists at least one common eigenvector $v\in V^{\C}$ for $\op{F}$ and $\op{G}$ with $\op{F}$-eigenvalue $i\mu$ $(\mu>0)$ and $\op{G}$-eigenvalue $i\gamma$ $(\gamma>0)$. Writing $v=x+iy$ with $x, y \in V$, we see from Lemma \ref{lem3} that $\op{F}x=-\mu y$, $\op{F}y=\mu x$, $\op{G}x=-\gamma y$ and $\op{G}y=\gamma x$, and that $W=span\{x, y\}$ is a symplectic plane in $V$. In fact, after the appropriate normalization of $v$, the vectors $x, y$ are a symplectic basis of $W$. The subspace $W$ is invariant under both $\op{F}$ and $\op{G}$, so its skew-orthogonal complement $W^\omega$ is also invariant under both $\op{F}$ and $\op{G}$. We can now induct over the dimension of $V$ similar to that in the proof of Williamson's theorem to find our desired basis.
\end{proof}

We can generalize this slightly to the positive-semi definite case.
\begin{theorem}\label{simul2}
Let $f$ and $g$ be positive-semi definite quadratic forms on $V$ such that the intersection of their radicals is a symplectic subspace. Then $f$ and $g$ Poisson commute iff there exists a symplectic basis $\{x_1, x_2, \dots, x_n, y_1, y_2, \dots, y_n\}$ of $V$ such that 
\begin{align*}
\op{F}x_j&= -\mu_j y_j \;\;\text{and}\;\; \op{F}y_j= \mu_j x_j,\\
\op{G}x_j&= -\gamma_j y_j \;\;\text{and}\;\; \op{G}y_j= \gamma_j x_j
\end{align*}
Here $\{\pm i\mu_1, \pm i\mu_2, \dots, \pm i\mu_k\}$ and $\{\pm i\gamma_1, \pm i\gamma_2, \dots, \pm i\gamma_l\}$ is a list of all non-zero eigenvalues of $\op{F}$ and $\op{G}$ respectively, where $\mu_j, \gamma_j>0$ and where each eigenvalue is repeated according to its multiplicity. For $v=\sum_{j=1}^{n}(s_j x_j +t_j y_j)\in V$, we have
\begin{equation}f(v)=\sum_{j=1}^{k}\mu_j (s_j^2+t_j^2), \;\;\; g(v)=\sum_{j=1}^{l}\gamma_j (s_j^2+t_j^2)\end{equation}
\end{theorem}
\begin{proof}
Suppose there exists such a basis. The $\op{F}$ and $\op{G}$ commute, so $f$ and $g$ Poisson commute by Lemma \ref{lem5}. Now suppose that $f$ and $g$ Poisson commute, then $\op{F}$ and $\op{G}$ commute by Lemma \ref{lem5}. So, there exists at least one common eigenvector $v\in V^{\C}$ for $\op{F}$ and $\op{G}$. If the $\op{F}$-eigenvalue is $i\mu$ $(\mu>0)$ and the $\op{G}$-eigenvalue is $i\gamma$ $(\gamma>0)$. Writing $v=x+iy$ with $x, y \in V$, we see from Lemma \ref{lem3} that $\op{F}x=-\mu y$, $\op{F}y=\mu x$, $\op{G}x=-\gamma y$ and $\op{G}y=\gamma x$, and that $W=span\{x, y\}$ is a symplectic plane in $V$. In fact, after the appropriate normalization of $v$, the vectors $x, y$ are a symplectic basis of $W$.\\
If only one of the eigenvalues is zero, then let us assume without loss of generality that the $\op{F}$-eigenvalue is $0$ and the $\op{G}$-eigenvalue is $i\gamma$ $(\gamma>0)$. Writing $v=x+iy$ with $x,y \in V$, we see from Lemma \ref{lem3} applied to $\op{G}$ that $\op{F}x=0$, $\op{F}y=0$, $\op{G}x=-\gamma y$ and $\op{G}y=\gamma x$, and that $W=span\{x, y\}$ is a symplectic plane in $V$. In fact, after the appropriate normalization of $v$, the vectors $x, y$ are a symplectic basis of $W$.\\
If both eigenvalues are zero, then $v\in Rad(A)\cap Rad(B)$. Take $x\in \Re v$ and since the intersection of radicals is a symplectic subspace, we can find a real $y\in Rad(A)\cap Rad(B)$ such that $\omega(x,y)=1$. $W=span\{x,y\}$ forms a symplectic plane and the vectors $x,y$ are a symplectic basis of $W$.
The subspace $W$ is invariant under both $\op{F}$ and $\op{G}$, so its skew-orthogonal complement $W^\omega$ is also invariant under both $\op{F}$ and $\op{G}$.\\
And now we repeat this procedure on $W^\omega$. We can now induct over the dimension of $V$ similar to that in the proof of Williamson's theorem to find our desired basis.
\end{proof}

Unfortunately, we cannot generalize this further to the case where $\op{F}$ and $\op{G}$ are not diagonalizable but admit a Jordan decomposition. This is because it is not guaranteed that two commuting operators can be simultaneously Jordan decomposed.\\
The following example will elucidate this further.
\begin{example}
Let $n\geq3$ and let $\op{J_n}$
be the $n$-th Jordan block, the $n\times n$ matrix whose entries are all $0$
except just above the diagonal where the $n-1$
entries equal $1$.
We claim that although $\op{J_n}$ obviously commutes with its square $\op{J^2_n}$, these matrices cannot be simultaneously Jordan decomposed.\\
Indeed any matrix $\op{P}$
Jordan decomposing $\op{J_n}$
will satisfy
$$\op{P^{-1}}\op{J_n}\op{P}=\op{J_n}$$
because of the uniqueness of Jordan forms.\\
But this will force (by squaring that equality)
$$\op{P^{-1}}\op{J^2_n}\op{P}=\op{J^2_n}$$
which is not in Jordan form.
Hence no matrix $\op{P}$
can simultaneously Jordan decompose both $\op{J_n}$
and $\op{J^2_n}$.
\end{example}

Let us now state some corollaries \footnote{the following corollaries are thanks to very fruitful discussions with Dr. Hemant Kumar Mishra, Department of Electrical and Computer Engineering, Cornell University, USA.} that find further useful applications in other fields.
\begin{corollary}\label{3.4}
A symmetric positive-definite matrix $\op{M_A}\in\mathbb{M}(\R^{2n})$ is orthosymplectically diagonalizable in the sense of Williamson's theorem if and only if $\op{J}\op{M_A}=\op{M_A}\op{J}$.
\end{corollary}
\begin{proof}
The condition $\op{J}\op{M_A}=\op{M_A}\op{J}$ means $\op{J}\op{M_A}$ and $\op{J}\op{I}$ commute with each other, where $\op{I}\in\mathbb{M}(\R^{2n})$ is the identity matrix. Since the symplectic eigenvalues of $\op{I}$ are all $1$, it follows that any symplectic matrix diagonalizing $\op{I}$ in the sense of Williamson's theorem is an orthogonal matrix. The statement of the corollary follows from Theorem \ref{simul}. 
\end{proof}

\begin{corollary}\label{simulcor}
Let $\mathcal{F}$ be the set of Poisson commuting positive-definite quadratic forms on $V$ i.e. if $f,g\in\mathcal{F}$ then $\{f,g\}=0$, then there exists a symplectic basis of $V$ which diagonalizes each element of $\mathcal{F}$ in the sense of Williamson's theorem.
\end{corollary}

\begin{theorem}
Let $\op{A},\op{B}\in\mathbb{M}(\R^{2n})$ be symmetric positive-definite matrices. If the matrices satisfy $[\op{A},\op{B}]=0$ as well as $[\op{JA},\op{JB}]=0$, then for all $s\in\R$, the matrices $\op{A^s}$ and $\op{B^s}$ satisfy $[\op{JA^s},\op{JB^s}]=0$.
\end{theorem}
\begin{proof}
Assume that $\op{A},\op{B}\in\mathbb{M}(\R^{2n})$ are symmetric positive-definite matrices satisfying $[\op{A},\op{B}]=0$ as well as $[\op{JA},\op{JB}]=0$. Then this implies:
\begin{equation}\label{eq3.6}\op{B^{-1}}\op{A}\op{J}=\op{J}\op{A}\op{B^{-1}}\end{equation}
Since $\op{A}$ and $\op{B}$ also commute, the matrix $\op{A}\op{B^{-1}}$ is a symmetric positive-definite matrix. Therefore, combining equation (\ref{eq3.6}) and Corollary \ref{3.4}, we get that $\op{A}\op{B^{-1}}$ is orthosymplectically diagonalized in the sense of Williamson's theorem. This implies that for any $s$, the matrix $\op{A^s}\op{B^{-s}}$ is also orthosymplectically diagonalized in the sense of Williamson's theorem. Invoking Corollary \ref{3.4} again, we thus have $\op{J}\op{A^s}\op{B^{-s}}=\op{A^s}\op{B^{-s}}\op{J}$. Using the commutativity of $\op{A}$ and $\op{B}$, this becomes $\op{A^s}\op{J}\op{B^s}=\op{B^s}\op{J}\op{A^s}$.
\end{proof}

Let us define $\op{A}\sharp_t\op{B}\coloneqq \op{A}^{1/2}(\op{A^{-1/2}}\op{B}\op{A^{-1/2}})^t\op{A}^{1/2}$ for all $t\in [0,1]$. Note that $\op{A}\sharp\op{B}\coloneqq\op{A}\sharp_{1/2}\op{B}$ is the geometric mean of $\op{A}$ and $\op{B}$.

\begin{corollary}
Let $\op{A},\op{B}\in\mathbb{M}(\R^{2n})$ be symmetric positive-definite matrices. \\Let $\mathcal{G}\coloneqq\{\op{A}\sharp_t\op{B}\;|\; t\in [0,1]\}$. If $\op{A}$ and $\op{B}$ are simultaneously diagonalizable by a symplectic matrix in the sense of Williamson's theorem, then the family $\mathcal{G}$ is also simultaneously diagonalizable by a symplectic matrix in the sense of Williamson's theorem.
\end{corollary}
\begin{proof}
Follows from the property $\op{M}^T(\op{A}\sharp_t\op{B})\op{M}=(\op{M}^T\op{A}\op{M})\sharp_t(\op{M}^T\op{B}\op{M})$ for any non-singular matrix $\op{M}$ (see \cite{lim}, Lemma 2.1, p. 1498-1514).
\end{proof}

\section{Applications}
\subsection{Decoupling physical systems}
Theorems \ref{simul} and \ref{simul2} can be used in physics to decouple a system comprising several Poisson commuting quadratic phase space observables (usually, Hamiltonians) such that the equations of motion remain invariant.\\ Let us give an example of this is in statistical thermodynamics.

Suppose we have a system of $N$ non-interacting particles in $\R^d$ whose quadratic symmetric positive-definite Hamiltonians form a Poisson commuting family. Then the configuration space is $Q=\R^{Nd}$ and the phase space is naturally given by the cotangent bundle \\$T^*Q = \R^{2Nd}$.\\ The canonical partition function for such a system is given by:
\begin{equation}\mathcal{Z}=\int_{\R^{2Nd}}\frac{1}{(2\pi\hbar)^{Nd}}\exp{-\beta\sum_{i=1}^N H_i(z_i,z_i)}\;d^d z_1\dots d^d z_N\end{equation}
where $\beta=\frac{1}{K_B T}$ and $K_B$ \& $T$ are the \textit{Boltzmann constant} and temperature respectively; \\$\hbar$ is the \textit{reduced Planck constant};\\ $H_i$ and $z_i=(q_i,p_i)\in\R^{2Nd}$ are the Hamiltonian and the phase space coordinates for the $i$-th particle respectively;\\
and $d^d z_i=d^d q_i\; d^d p_i$ is the $i$-th component of the phase space volume element.

By our assumption, $\{H_i, H_j\}=0$ for all $i,j$. Then from Corollary \ref{simulcor}, we can find symplectic coordinates for $\R^{2Nd}$ such that the $H_i$ are simultaneously diagonalized for all $i$ in the sense of Williamson's theorem.\\
Let us call these coordinates $\{x_1,\dots,x_{Nd},y_1,\dots,y_{Nd}\}$. Then if $\lambda_{i,k}$ denotes the symplectic eigenvalue of $H_i$ corresponding to the $k$-th symplectic eigenvector pair i.e. $\{x_k,y_k\}$ and $l=(i-1)d+j$, we have:
\begin{align*}
\mathcal{Z}&=\frac{1}{(2\pi\hbar)^{Nd}}\prod_{i=1}^N\prod_{j=1}^{d}\int_{\R^2}\exp{-\beta\lambda_{i,l}(x_{l}^2+y_{l}^2)}\;dx_{l}\;dy_{l}\\
&=\frac{1}{(2\pi\hbar)^{Nd}}\prod_{i=1}^N\prod_{j=1}^{d}\left(\frac{\pi}{\beta}\right)\frac{1}{\lambda_{i,l}}\\
&=\frac{1}{(2\hbar\beta)^{Nd}}\prod_{i=1}^N\frac{1}{\sqrt{det(H_i)}}
\end{align*}
Going from the first line to the second, we have use the value of the Gaussian integral (see Appendix \ref{6.3}) and going from the second line to the third, we have used the fact that determinant of a symplectic matrix is unity.

Now, we consider a system of $N$ interacting particles in $\R^d$ whose quadratic symmetric positive-definite Hamiltonians form a Poisson commuting family. The phase space is again $\R^{2Nd}$. Using the same notation as before, the canonical partition function is given by:
\begin{equation}\mathcal{Z}_{int}=\int_{\R^{2Nd}}\frac{1}{(2\pi\hbar)^{Nd}}\exp{-\beta\sum_{i=1}^N H_i(z,z)}\;d^d z_1\dots d^d z_N\end{equation}
where $z=(z_1,z_2,\dots,z_N)\in\R^{2Nd}$

By our assumption, $\{H_i, H_j\}=0$ for all $i,j$. Then from Corollary \ref{simulcor}, we can find symplectic coordinates for $\R^{2Nd}$ such that the $H_i$ are simultaneously diagonalized for all $i$ in the sense of Williamson's theorem.\\
Let us call these coordinates $\{x_1,\dots,x_{Nd},y_1,\dots,y_{Nd}\}$. Then if $\lambda_{i,k}$ denotes the symplectic eigenvalue of $H_i$ corresponding to the $k$-th symplectic eigenvector pair, we have:
\begin{align*}
\mathcal{Z}_{int}&=\frac{1}{(2\pi\hbar)^{Nd}}\prod_{j=1}^{Nd}\int_{\R^2}\exp{-\beta\sum_{i=1}^N\lambda_{i,j}(x_{j}^2+y_{j}^2)}\;dx_{j}\;dy_{j}\\
&=\frac{1}{(2\hbar\beta)^{Nd}}\prod_{j=1}^{Nd}\frac{1}{\sum_{i=1}^N\lambda_{i,j}}\\
\end{align*}

\subsection{Symplectic diagonalization and phase space constraints}
We will now try to establish a geometric connection between the degenerate (symmetric positive-semi definite) and
non-degenerate cases (symmetric positive-definite) in context of symplectic diagonalization from the perspective of phase space constraints.

Let $\operatorname{M}$ be a real symmetric positive-semi definite automorphism on
$\mathbb{R}^{2n}$ whose kernel is a symplectic subspace such that its non-zero symplectic eigenvalues are $\{\lambda_{j}\}$ for $1\leq j\leq k$

Consider the Hamiltonian $H(z')=\langle z',\operatorname{M}z'\rangle$ for
$z'\in\mathbb{R}^{2n}$.

Let $\operatorname{\tilde{M}}$ be a real symmetric positive-definite automorphism
on $\mathbb{R}^{2n}$ such that its symplectic eigenvalues are $\{\tilde{\lambda}_{j}\}$ for
$1\leq j\leq n$ and 
\begin{align*}
\tilde{\lambda}_{i} & =\lambda_{i};\;\;1\leq i\leq k\\
\end{align*}

Consider the Hamiltonian $\tilde{H}(z')=\langle z',\operatorname{\tilde{M}%
}z'\rangle$ for $z'\in\mathbb{R}^{2n}$.\newline Suppose $[\op{J}\op{M},\op{J}\op{\tilde{M}}]=0$. \\Then we know from Theorem \ref{simul2} that there exists a symplectic matrix $\op{S}$ that will diagonalize $\op{M}$ and $\op{\tilde{M}}$ simultaneously in the sense of Williamson's theorem.  More explicitly,
\[
\tilde{H}(\op{S}z)=\langle\operatorname{S}z,\operatorname{\tilde{M}}\operatorname{S}z\rangle=\sum_{i=1}^{n}\tilde{\lambda}_{i}(x_{i}^{2}+p_{i}^{2})
\]
\[
H(\op{S}z)=\langle\operatorname{S}z,\operatorname{M}\operatorname{S}z\rangle=\sum_{i=1}%
^{k}\lambda_{i}(x_{i}^{2}+p_{i}^{2})
\]
where $z=\op{S}^{-1}z'$.

Comparing both the decompositions and using the assumption that
$\tilde{\lambda}_{i}=\lambda_{i}$ for $1\leq i\leq k$, we can define the following constraints:
\begin{align*}
x_{j}  &  =0;\;\;(k+1)\leq j\leq n\\
p_{i}  &  =0;\;\;(k+1)\leq i\leq n
\end{align*}
We can collectively denote these by $\chi_{j}=0$ for $1\leq j\leq
2(n-k)$.\newline We shall call $\chi_{j}$ the \textit{H\"{o}rmander
constraints}.\newline

Let $\Gamma$ be the H\"{o}rmander constraint surface generated from $\chi_{j}=0$, then
\begin{align*}
z^{T}\operatorname{S}^{T}\operatorname{\tilde{M}}\operatorname{S}%
z\vert_{\Gamma}  &  =z^{T}\operatorname{S}^{T}\operatorname{M}\operatorname{S}%
z\\
\Rightarrow\tilde{H}(\op{S}z)\vert_{\Gamma}  &  =H(\op{S}z)
\end{align*}
This can be expressed as,
\[
\tilde{H}(\op{S}z)=H(\op{S}z)+\sum_{j=1}^{2(n-k)}\chi_{j}C_{jj}\chi_{j}%
\]
for $C_{jj}\in\mathbb{R}$\newline\newline Since the $\operatorname{S}%
^{-1}z'=z=(x_{1},\dots,x_{n},p_{1},\dots,p_{n})\newline=(x_{1}%
,\dots,x_{(k)},\chi_{1},\dots,\chi_{(n-k)},p_{1},\dots,p_{k}%
,\chi_{(n-k+1)},\dots,\chi_{(2n-2k)})$ is a symplectic basis,\\ the set of
H\"{o}rmander constraints $\{\chi_{j}\}$ is a set of \textit{second-class constraints},
using the Dirac-Bergmann terminology (See \cite{ehlers}).


We can now summarize the above arguments into the following

\begin{proposition}\label{constr}
Let $\operatorname{M}\geq 0$ and $\operatorname{\tilde{M}}>0$ be two symmetric matrices where the former has a symplectic kernel
satisfying $[\op{J}\op{M},\op{J}\op{\tilde{M}}]=0$
that decompose, for an $\operatorname{S}%
\in\operatorname{Sp}(n)$, into
\begin{align*}
H(\operatorname{S}z)  &  =\langle\operatorname{S}z,\operatorname{M}\operatorname{S}z\rangle=\sum
_{i=1}^{k}\lambda_{i}(x_{i}^{2}+p_{i}^{2})
\end{align*}
for $1\leq k\leq n$ and
\begin{align*}
\tilde{H}(\operatorname{S}z)  & =\langle\operatorname{S}z,\operatorname{\tilde{M}}%
\operatorname{S}z\rangle=\sum_{i=1}^{n}\tilde{\lambda}_{i}(x_{i}^{2}+p_{i}^{2})
\end{align*}
with $\lambda_{i}=\tilde{\lambda}_{i}$ for $1\leq i\leq k$ \newline\newline Then, the
Hamiltonian $H$ can be extended off the H\"{o}rmander constraint surface
$\Gamma$, in phase space, such that $\tilde{H}(\op{S}z)=H(\op{S}z)+\sum_{j=1}^{2(n-k)}\chi_{j}C_{jj}\chi_{j}$ where $\chi_{j}$ are the
H\"{o}rmander constraints and $C_{jj}\in\mathbb{R}$.
\end{proposition}

Before we try to apply Proposition \ref{constr} to phase space cylinders and ellipsoids, it would be in our favour to state the definition of \textit{symplectic capacity}, listing its properties which would naturally prove useful to us.

\begin{definition}
[Symplectic Capacity \cite{Birk,HZ,PR}]A symplectic capacity on $(\mathbb{R}^{2n},\omega)$
is a mapping, which to every subset $\Omega$ of $\mathbb{R}^{2n}$,
associates a number $c_{lin}(\Omega)\geq0$, or $\infty$ having the following
properties:

\begin{itemize}
\item $c(\Omega)\leq c(\Omega^{\prime})$ if $\Omega\subset\Omega^{\prime}$

\item $c(f(\Omega))=c(\Omega)$ for every symplectomorphism $f$

\item $c(\lambda\Omega)=\lambda^{2}c(\Omega)$ for every $\lambda\in\mathbb{R}$

\item $c(B(R))=c(Z_{j}(R))=\pi R^{2}$ where $B(R):\sum_{i=1}^{n}(x_{i}%
^{2}+p_{i}^{2})\leq R$ and $Z_{j}(R):x_{j}^{2}+p_{j}^{2}\leq R$ are the phase
space ball and cylinder of radius $R$ respectively.
\end{itemize}
\end{definition}

Let us try to define phase space cylinders and ellipsoids in the context of real symmetric positive-definite and semi definite matrices.

\begin{definition}\label{cyl}
    A phase space cylinder of
radius $R$, $Z_{j}(R)$, can be described as being a diagonal real symmetric
matrix $\operatorname{M}_j\geq 0$ whose kernel is a symplectic space with dimension
$2n-2$ such that the only non-zero entries are at $j,n+j$, both equal to
$R^{-2}$ such that
\[
Z_{j}(R):\langle z,\operatorname{M}_jz\rangle\leq1\Rightarrow x_{j}^{2}+p_{j}^{2}\leq R^{2}%
\]
\end{definition} \par
\begin{definition}\label{ell}
    A phase space ball of radius $R$, $B(R)$, can be described as being a
diagonal real symmetric matrix $\operatorname{\tilde{M}}>0$, all of whose
entries are equal to $R^{-2}$ such that
\[
B(R):\langle z,\operatorname{\tilde{M}}z\rangle\leq1\Rightarrow\sum_{i=1}^{n}(x_{i}%
^{2}+p_{i}^{2})\leq R^{2}%
\]
\end{definition}

Since $\operatorname{M}_j,\operatorname{\tilde{M}}$ satisfy $[\op{J}\op{M}_j,\op{J}\op{\tilde{M}}]=0$, they are
simultaneously diagonalizable in the sense of Williamson and from Proposition \ref{constr}, it follows that
\[
\langle z,\operatorname{\tilde{M}}z\rangle=\langle z,\operatorname{M}_jz\rangle+\sum_{l=1}^{2(n-k)}\chi_{l}C_{ll}\chi_{l}%
\]\\
In other words, $B(R)$ is the extension of $Z_{j}(R)$ off the appropriate
H\"{o}rmander constraint surface.\newline

Now from $\langle z,\operatorname{\tilde{M}}z\rangle\leq1$, we get $B(R)$ (from definition \ref{ell}) and from (the indices imply summation)

$\langle z,\operatorname{M}_j z\rangle +\chi_{l}C_{ll}\chi_{l}\leq1$, we get $Z_{j}(R\sqrt{|1-\chi_{l}C_{ll}%
\chi_{l}|})$ (from definition \ref{cyl}).\par If we can find a symplectic embedding: $\Phi:B(R)\to Z_{j}(R\sqrt{|1-\chi_{l}C_{ll}\chi_{l}|})$,
then \vspace{0.3cm}\\from \textit{Gromov's symplectic non-squeezing theorem} we must have $R\leq R\sqrt{|1-\chi_{l}C_{ll}\chi_{l}|}$ i.e. $2\leq\chi_{l}C_{ll}\chi_{l}\leq 0$. \par We may call the saturation of the previous inequality: $\chi_{l}C_{ll}\chi_{l}=2$, the \textit{Gromov hypersurface}.

\section{Conclusion}
Williamson's symplectic diagonalization for symmetric positive-semi definite matrices has many potential applications in quantum information and computing, where covariance matrices play an important role, as an extension to the readily used Williamson normal form.\par
The uses of Theorem \ref{simul} and its corollary, as well as Proposition \ref{constr} still require further investigation, since they may be repurposed as lemmas for newer results in phase space topology, uncertainty relations and the analysis of Hamiltonian systems at their equilibrium points by using the Hessian as a quadratic function on phase function, for stability analysis.

\section{Appendix}\label{appen}
\subsection{Corollary \ref{semi} for $V=\R^{2n}$}
We state and prove a special case of Theorem \ref{horm1} for $V=\R^{2n}$.

\begin{corollary}
Let $\operatorname{M}$ be a real symmetric positive-semi definite $2n\times2n$
matrix \\whose kernel is a
symplectic subspace of $\mathbb{R}^{2n}$ of dimension $2(n-k)$.\par Then there
exists a matrix $S\in\operatorname*{Sp}(\R^{2n})$ such that $\operatorname{S}%
^{T}\operatorname{M}\operatorname{S}=%
\operatorname{diag}(\Lambda,\Lambda)
$ with\\$\operatorname{\Lambda}=\operatorname*{diag}(\lambda_{1},\dots
,\lambda_{n})$ being a diagonal matrix such that $0<\lambda_{1} \leq\cdots
\leq\lambda_{k}$ and\\ $\lambda_{(k+1)}=\cdots=\lambda_{n}=0$.\par More
explicitly,
\[
\operatorname{S}(x,p)^{T}\circ\operatorname{M}\circ\operatorname{S}%
(x,p)=\sum_{j=1}^{k}\lambda_{j}(x_{j}^{2}+p_{j}^{2})
\]

\end{corollary}

\begin{proof}
Let $\operatorname{N}=\ker\operatorname{M}$ with dimension $2m=2(n-k)$%
.\newline Since $\operatorname{N}$ is a symplectic subspace, it admits a
symplectic basis $\{\vec{q}_{i},\vec{q}_{(m+i)}\}_{1\leq i\leq(n-k)}$. \\Let
$\operatorname{B}=\mathbb{R}^{2n}/\operatorname{N}$, we can always find a
basis for the $2k$ dimensional subspace $\operatorname{B}$, call it $\{
\vec{b}_{i}\}_{1\leq i\leq2k}$.

Consider the matrix,
\[
\operatorname{W}=%
\begin{bmatrix}
\vert & \vert &  & \vert & \vert & \vert &  & \vert\\
\vec{b}_{1} & \vec{b}_{2} & \cdots & \vec{b}_{2k} & \vec{q}_{1} & \vec{q}_{2}
& \cdots & \vec{q}_{2(n-k)}\\
\vert & \vert &  & \vert & \vert & \vert &  & \vert\\
&  &  &  &  &  &  &
\end{bmatrix}
\]

Note that ${\vec{q}_{i}}\,^{T}\operatorname{M}\vec{q}_{j}=\op{0}_{ij}$ since $\vec
{q}_{i},\vec{q}_{j}\in N$

$\Rightarrow{\vec{q}_{i}}\,^{T}\operatorname{M}\vec{b}_{j}=\op{0}_{ij}$ and 
${\vec{b}_{i}}\,^{T}\operatorname{M}\vec{b}_{j}=\tilde{\operatorname{M}}_{ij}%
>0$\newline

Consequently,
\[
\operatorname{W}^{T}\operatorname{M}\operatorname{W}=
\begin{bmatrix}
\tilde{\operatorname{M}}_{2k\times2k} & \\
& \operatorname{0}_{2m\times2m}%
\end{bmatrix}
\]
Note that the $2k\times2k$ sub-matrix $\tilde{M}$ is real symmetric
positive-definite. \par Therefore, by Williamson's theorem, there exists
$\operatorname{\tilde{S}}\in\operatorname{Sp}(\R^{2k})$ such that \\%
\[
\operatorname{\tilde{S}}^{T}\operatorname{\tilde{M}}\operatorname{\tilde{S}}=
\begin{bmatrix}
\tilde{\operatorname{\Lambda}} & 0\\
0 & \tilde{\operatorname{\Lambda}}%
\end{bmatrix}
\]
where $\tilde{\Lambda}=\operatorname{diag}(\tilde{\lambda}_{1},\dots
,\tilde{\lambda}_{k})$ such that $\pm i\tilde{\lambda}_{j}$ are the
eigenvalues of $\operatorname{J_{2k}}\tilde{\operatorname{M}}$\newline

The columns of $\operatorname{\tilde{S}}$ are the symplectic eigenvectors, let
us arrange them in a way such that $0<\tilde{\lambda}_{1}\leq\dots\leq
\tilde{\lambda}_{k}$, and label them$\{\vec{p}_{i},\vec{p}_{(k+i)}\}_{1\leq
i\leq k}$ (this is a basis for $\operatorname{B}$).\newline

Now consider,
\[
S=
\begin{bmatrix}
\vert &  & \vert & \vert &  & \vert & \vert &  & \vert & \vert &  & \vert\\
\vec{s}_{1} & \cdots & \vec{s}_{k} & \vec{q}_{1} & \cdots & \vec{q}_{(n-k)} &
\vec{s}_{(k+1)} & \cdots & \vec{s}_{2k} & \vec{q}_{(n-k+1)} & \cdots & \vec
{q}_{2(n-k)}\\
\vert &  & \vert & \vert &  & \vert & \vert &  & \vert & \vert &  & \vert\\
&  &  &  &  &  &  &  &  &  &  &
\end{bmatrix}
\]
where $\vec{s_{i}}=%
\begin{bmatrix}
\vert\\
\vec{p}_{i}\\
\vert\\
\operatorname{0}_{m\times1}\\
\end{bmatrix}
$\newline

Note that $\vec{q}_{i}\,^{T}\operatorname{M}\vec{s}_{j}=\op{0}_{ij}$ and
$\vec{s}_{i}\,^{T}\operatorname{M}\vec{s}_{j}=\tilde{\lambda}_{i}\delta_{ij}$,
as guaranteed by Williamson's theorem.\newline

Finally,
\[
\operatorname{S}^{T}\operatorname{M}\operatorname{S}=
\begin{bmatrix}
\Lambda & 0\\
0 & \Lambda
\end{bmatrix}
\]
\newline where $\Lambda=\operatorname{diag}(\lambda_{1},\dots,\lambda_{n})$
with $0<\lambda_{1}=\tilde{\lambda}_{1} \leq\cdots\leq\lambda_{k}%
=\tilde{\lambda}_{k}$ and $\lambda_{(k+1)}=\cdots=\lambda_{n}=0$\newline

Since, $\tilde{\operatorname{S}}$ is a symplectic matrix, its columns form a
symplectic basis such that $\vec{s}_{i}\,^{T}\operatorname{J_{2n}}\vec{s}%
_{j}=\delta_{j,k+i}$\par
Following the proof of Williamson's theorem in
[2], we see that $\vec{s}_{i}=\tilde{\lambda}_{i}^{-1/2}\vec{s^{\prime}}_{i}$\\
where $\{\vec{s^{\prime}}_{j}\pm i\vec{s^{\prime}}_{(k+j)}\}_{1\leq j\leq k}$
are eigenvalues of $\operatorname{K}=-\tilde{\operatorname{M}}^{-1}\op{J_{2k}}$ and
$\operatorname{K}^{-1}=\operatorname{J_{2k}}\tilde{\operatorname{M}}$\newline such
that $\operatorname{K}^{-1}\vec{s^{\prime}}_{j}=\tilde{\lambda}^{-1}%
\vec{s^{\prime}}_{(k+j)}$ and $\operatorname{K}^{-1}\vec{s^{\prime}}%
_{(k+j)}=-\tilde{\lambda}^{-1}\vec{s^{\prime}}_{j}$ for $1\leq j\leq k$.\par
Therefore,%
\[
\vec{q}_{i}\,^{T}\operatorname{J_{2n}}\vec{s}_{j}=\tilde{\lambda}^{1/2}\vec{q}%
_{i}\,^{T}\operatorname{J_{2n}}^{2}\operatorname{M}\vec{s}_{(k+j)}=-\tilde{\lambda
}^{1/2}\vec{q}_{i}\,^{T}\operatorname{M}\vec{s}_{j}=0.
\]
\newline Similarly, $\vec{q}_{i}\,^{T}\op{J_{2n}}\vec{s}_{(k+j)}=0$%
\[
\Rightarrow\operatorname{S}^{T}\operatorname{J_{2n}}\operatorname{S}=%
\begin{bmatrix}
\operatorname{0}_{n\times n} & \operatorname{I}_{n\times n}\\
-\operatorname{I}_{n\times n} & \operatorname{0}_{n\times n}%
\end{bmatrix}
=\operatorname{J_{2n}}%
\]
\par This shows that $\operatorname{S}\in\operatorname{Sp}(\R^{2n})$.
\par And finally,\[%
\begin{bmatrix}
x_{1} & \cdots & x_{n} & p_{1} & \cdots & p_{n}%
\end{bmatrix}
\operatorname{S}^{T}\circ\operatorname{M}\circ\operatorname{S}%
\begin{bmatrix}
x_{1}\\
\vdots\\
x_{n}\\
p_{1}\\
\vdots\\
p_{n}%
\end{bmatrix}
=\sum_{j=1}^{k}\lambda_{j}(x_{j}^{2}+p_{j}^{2})
\]

\end{proof}

\subsection{Examples}

Here, we shall give examples of symplectic eigenvalue computation. \\Note that the blank entries in the matrices below are all zero.

i)
$\operatorname{Q}=%
\begin{bmatrix}
5 & 3\\
3 & 2
\end{bmatrix}
$\par$\operatorname{J_{2}Q}=%
\begin{bmatrix}
3 & 2\\
-5 & -3
\end{bmatrix}
$ has eigenvalues $\pm i$.

Therefore, the symplectic eigenvalue is $1$.

\vspace{0.5cm}
ii) $\operatorname{Q}=%
\begin{bmatrix}
6 & 0 &  & \\
0 & 3 &  & \\
&  & 3 & -1\\
&  & -1 & 1
\end{bmatrix}
$\par$\operatorname{J_{4}Q}=%
\begin{bmatrix}
&  & 3 & -1\\
&  & -1 & 1\\
-6 & 0 &  & \\
0 & -3 &  &
\end{bmatrix}
$ has eigenvalues $\pm i\sqrt{\frac{3}{2}(7+\sqrt{33})}$, $\pm i\sqrt{\frac
{3}{2}(7-\sqrt{33})}$

Therefore, the symplectic eigenvalues are
$\sqrt{\frac{3}{2}(7+\sqrt{33})},\sqrt{\frac{3}{2}(7-\sqrt{33})}$.

\subsection{Evaluation of the Gaussian integral}\label{6.3}
We will demonstrate here how to evaluate the following integral for any real $\alpha>0$
$$I=\int_{-\infty}^{\infty}e^{-\alpha x^2}\; dx$$
Consider the following one parameter family:
$$I(t)=\int_{-\infty}^{\infty}\frac{e^{-\alpha x^2}}{1+(x/t)^2}\;dx$$
After the simple substitution $u=x/t$, we have:
$$I(t)=t\int_{-\infty}^{\infty}\frac{e^{-\alpha u^2 t^2}}{1+u^2}\;du$$
Note that $\displaystyle\lim_{t\to 0}\frac{I(t)}{t}=\int_{-\infty}^{\infty}\frac{1}{1+u^2}\;du=\pi$\;; and $\displaystyle I(\infty)=\lim_{t\to\infty}I(t)=I$.

We now have:
$$\frac{e^{-\alpha t^2}}{t} I(t)=\int_{-\infty}^{\infty}\frac{e^{-\alpha t^2(1+u^2)}}{1+u^2}\;du$$
Differentiating both sides with respect to $t$ gives us
$$\dv{t}(\frac{e^{-\alpha t^2}}{t} I(t))=-2\alpha t e^{-\alpha t^2}\int_{-\infty}^{\infty}e^{-\alpha u^2 t^2}\;du=-2\alpha e^{-\alpha t^2} I(\infty)$$
Now integrating with respect to $t$ from $0$ to $\infty$ and using the symmetry of the integrand on the right hand side 
$$\lim_{t\to\infty}\frac{e^{-\alpha t^2}}{t} I(t)-\lim_{t\to 0}\frac{e^{-\alpha t^2}}{t} I(t)=-\pi=-\alpha I^2(\infty)$$
Finally, 
$$\int_{-\infty}^{\infty}e^{-\alpha x^2}\;dx=\sqrt{\frac{\pi}{\alpha}}$$
\section*{Acknowledgement}
I would like to thank Prof. Maurice de Gosson, Faculty of Mathematics,
Universit\"{a}t Wien, Austria foremost for taking me on for a summer project when all other doors seemed closed to me, for suggesting this interesting project topic, and for helping me write the introduction and background sections.\\ Secondly, I would like to thank my senior thesis advisor, Prof. Reyer Sjamaar, Department of Mathematics, Cornell University, USA for mentoring, keenly advising, and being thorough while helping me tighten my mathematical arguments, especially in the re-derivation of Hörmander's results in the background section and in the main theorem of the results section.\\
Finally, I would like to thank Dr. Hemant Kumar Mishra, Department of Electrical and Computer Engineering, Cornell University, USA for his valuable inputs in the corollaries of the results section and his literature recommendations that helped me see different perspectives with respect to symplectic eigenvalues and eigenvector pairs.\\
Without all the academic experience and mentoring of the aforementioned people, this work (with several potential future applications) would not have come to fruition.

\end{document}